\documentclass[fpartfrac]{myclass}

\algrenewcommand\algorithmicrequire{\textbf{Requires:}}
\algrenewcommand\algorithmicensure{\textbf{Outputs:}}

\newcommand\h{$h$\nobreakdashes}

\usepackage{semigroups}

\newcommand\titlefrac[2]{\multicolumn1c{$\dfrac{#1}{#2}$}}%

\usepackage[hidelinks]{hyperref}

\title{Families of Eliahou semigroups\\linked to Farey intervals}
\author{Axel Bacher}

\date{December 1, 2026}

\begin{document}

\maketitle

\begin{abstract}
We describe new classes of numerical semigroups with a Farey interval as a
crucial parameter. We compute their Eliahou number and show many cases in
which it is negative, encompassing families described by Delgado, Eliahou
and Fromentin, and Bras-Amorós. We show that all these semigroups satisfy
Wilf's conjecture.

This work is based on an experimental search that found all Eliahou-negative
semigroups with conductor up to~200 (and conjecturally up to~400) thanks to a
new way of representing the semigroup and pruning off unwanted branches. This
proves that Wilf's conjecture holds for conductor up to~200.
\end{abstract}

\section{Introduction}

A \emph{numerical semigroup} is a co-finite sub-semigroup of~$(\mathbb N, +)$.
Given some integers~$0 < \gamma_0 < \dotsb < \gamma_{\ell-1} < c$ such that
no~$\gamma_i$ can be written as a positive combination of the others, we use
the classical notation:
\[S = \langle\gamma_0,\dotsc,\gamma_{\ell-1}\rangle_c\]
to mean the semigroup generated by~$\gamma_0,\dotsc,\gamma_{\ell-1}$ and all
integers starting from~$c$.
This construction is unique if~$c-1\notin S$. In this case, we
say that~$S$ is \emph{canonical}; if~$c-1\in S$, we say that~$S$ is
\emph{non-canonical} (it is more properly a non-canonical semigroup
presentation).

We call~$\gamma_0,\dotsc,\gamma_{\ell-1}$ the \emph{left generators} and~$c$
the \emph{conductor} of~$S$ (the value~$c-1$ is the \emph{Frobenius number}
of~$S$). Primitive elements greater than or equal to~$c$ are called
\emph{right generators} and we denote by~$r$ their number; the total number of
generators~$e = \ell + r$ is called the \emph{embedding dimension} of~$S$. We
call \emph{rank} of~$S$, and denote by~$k$, the number of left elements
(elements less than~$c$). The first generator~$m = \gamma_0$ plays a special
role and is called the \emph{multiplicity}. The number~$g = c - k$ is the
number of gaps in~$S$ and called the \emph{genus}. These parameters are
illustrated in Figure~\ref{fig:full} in the next section.

Wilf~\cite{wilf} gives a simple conjecture concerning these parameters. Define
the \emph{Wilf number} of~$S$ as:
\begin{equation} \label{W} W = ke - c\text.\end{equation}
We say that~$S$ is \emph{Wilf-positive} if~$W\ge0$ and \emph{Wilf-negative}
otherwise. Wilf's conjecture states that every numerical semigroup is
Wilf-positive.
This conjecture is nearly fifty years old and
still open. It has been proven to be true in some specific cases,
namely~$e\le3$~\cite{froberg-gottlieb-haggkvist}, $e\ge
m/3$~\cite{eliahougraph}, $c\le3m$~\cite{eliahou}, $e\ge m/4$ when~$m$
divides~$c$~\cite{eliahou2}, $k\le12$ and~$m\le19$ (see~\cite{survey} for a
survey as of~2019).

Eliahou~\cite{eliahou} gives a sufficient condition for a semigroup to satisfy
Wilf's conjecture. It is easily seen that the right generators of~$S$ lie in
the interval~$[c,c+m)$. Let~$s = m-r$ be the number of elements in this
interval that are not primitive. Also write~$c = qm - \rho$ for~$0\le\rho < m$
(the number~$q$ is called the \emph{depth} of~$S$).  The \emph{Eliahou number}
of~$S$ has the two equivalent definitions:
\begin{align} \label{Er}
E &= k\ell + qr - c\text,\\ \label{Es}
E &= k\ell - qs + \rho\text.
\end{align}
We say that~$S$ is \emph{Eliahou-positive} if~$E\ge0$ and
\emph{Eliahou-negative} otherwise.
Since~$k\ge q$ and~$e = \ell+r$, we have~$E\le W$ and thus any
Eliahou-positive semigroup is Wilf-positive. Understanding Eliahou-negative
semigroups thus appears useful to understanding Wilf's conjecture.

Eliahou-negative semigroups are very rare. The first five,
namely~$\langle14,22,23\rangle_{56}$, $\langle16,25,26\rangle_{64}$,
$\langle17,26,28\rangle_{68}$, $\langle17,27,28\rangle_{68}$
and~$\langle18,28,29\rangle_{72}$, were discovered by Fromentin by an
exhaustive search of all semigroups with~$g\le60$. This was later expanded to
reach semigroups with~$g\le
100$~\cite{bras-amoros-rodriguez,fromentin-hivert,bras-amoros,delgado-eliahou-fromentin,delgado-trimming}.
These searches are based on an exploration of the \emph{numerical semigroup
tree}, described in~\cite{rosales-garcia-sanchez}. Delgado~\cite{delgado},
Eliahou and Fromentin~\cite{eliahou-fromentin} and
Bras-Amorós~\cite{bras-amoros} describe infinite families of Eliahou-negative
semigroups; notably, \cite{delgado} and~\cite{eliahou-fromentin} show that
there are semigroups with~$E$ taking an arbitrary value in~$\mathbb Z$.
Almirón and Moyano-Fernández~\cite{almiron-moyano-fernandez} also found some
semigroups not belonging to these families.

For this work, we computed all Eliahou-negative semigroups up to
conductor~$200$, which we expanded to~$400$ with some heuristics, using a
variant of Bras-Amorós's seeds
algorithm~\cite{bras-amoros-fernandez,bras-amoros}. All these semigroups are
Wilf-positive.
Based on this list, we define the classes of \emph{\h-regular} semigroups and
\emph{collision-free} semigroups. All the semigroups described
in~\cite{delgado,eliahou-fromentin,bras-amoros} are \h-regular and
collision-free, as well as every Eliahou-negative semigroup with conductor up
to~$218$; the ones in~\cite{almiron-moyano-fernandez} are nearly so, as are
the rest of the semigroups we found.

Our classification uses two key ingredients. The first is the notion of
\emph{$B_h$~set}, already used by Eliahou and
Fromentin~\cite{eliahou-fromentin}. If~$\Delta\subset\mathbb N$ is a finite
set and~$h\ge0$, we denote by~$h\Delta$ the \emph{\h-fold sumset} of~$\Delta$:
\[h\Delta = \{\delta_1 + \dotsb + \delta_h \mid \delta_1,\dotsc,\delta_h\in
\Delta\}\text.\]
We denote by:
\[\mchoose{\abs{\Delta}}{h} = \binom{\abs{\Delta} + h - 1}{h}\]
the number of ways to take~$h$ elements of~$\Delta$ with repetitions allowed.
If the sums of~$h\Delta$ are pairwise distinct up to permutation, the
set~$\Delta$ is called a \emph{$B_h$~set}. This is equivalent
to~$\abs{h\Delta} = \smchoose{\abs\Delta}{h}$. Relevant ways to
construct~$B_h$~sets are found e.g.\ in~\cite{kovacevic,obryant}.

The second is \emph{\h-Farey intervals}. Given an integer~$h\ge1$, an
\emph{\h-Farey fraction} is a reduced fraction of the form~$a/b$ where~$b\le
h$. Two consecutive \h-Farey fractions form an \emph{\h-Farey interval} of the
form~$(a'/b', a/b]$. It is known that~$ab' = a'b + 1$, which shows
that~$b$ and~$b'$ are coprime and that the width of the interval
is~$\frac1{bb'}$. Moreover, we have~$b + b' \le h$ since the
fraction~$\smash{\frac{a+a'}{b+b'}}$ lies in between. Farey fractions
are discussed in greater detail in \cite[Chapter~4]{concrete}.

The paper is organized as follows.
In Section~\ref{sec:preliminaries}, we give some preliminary definitions and
results.
In Section~\ref{sec:regular}, we define \h-regular and collision-free
semigroups, prove some of their properties and show how to construct them.
In Section~\ref{sec:numbers}, we give formulas for the Eliahou number of such
semigroups and show that they are all Wilf-positive.
In Section~\ref{sec:negative}, we discuss conditions under which these
semigroups are Eliahou-negative, construct several families and show how the
families already published fit into them.
In Section~\ref{sec:experimental}, we explain how we searched for
Eliahou-negative semigroups experimentally and prove that Wilf's conjecture
holds for~$c\le200$.

\section{Preliminaries} \label{sec:preliminaries}

\subsection{Left elements and threshold non-generators}

Throughout the paper, let~$S = \langle m,\Gamma\rangle_c$ be a semigroup,
where~$\Gamma = \{\gamma_1,\dotsc,\gamma_{\ell-1}\}$ is
the set of left generators except~$m$. The interval~$[c,c+m)$, called the
\emph{threshold interval}, plays a special role in relation to the Wilf and
Eliahou numbers, as recognized by Eliahou~\cite{eliahou}. We call
\emph{threshold non-generators} the elements of the threshold interval that
are not generators; their number is~$m-r$, which we denoted by~$s$.
Understanding the structure of the left elements and threshold non-generators
of~$S$ is key to understanding~$S$ and computing its Wilf and Eliahou numbers.

The left elements of~$S$ are of the form~$\lambda + jm$, where~$\lambda$
is in a sumset~$i\Gamma$ for~$i\ge0$ and~$0\le j <
\ceil{\sfrac{c-\lambda}{m}}$. The threshold non-generators are of the
form~$\lambda + \ceil{\sfrac{c-\lambda}{m}}m$ for~$\lambda < c+m$. In both
cases, we say that~$\lambda$ \emph{determines} these elements. An element can
be determined by several values~$\lambda$ if they are congruent modulo~$m$.
Examples are shown in Figure~\ref{fig:full}, in a manner resembling
that of Delgado~\cite{delgado}.

\begin{figure}[ht]%
\begin{tikzpicture}[box semigroups]
\draw[sgline] (-.55,.55) rectangle ++(14.1,-5.1);
\def\sgm{14}
\def\sgc{56}
\sgbox00
\foreach \x in {22,23} { \sgbox1{\x} }
\foreach \x in {44,45,46} { \sgbox2{\x} }
\foreach \x in {66,67,68,69} { \sgbox3{\x} }
\foreach \x in {14,28,42,56,36,37,50,51,64,65,58,59,60} { \sgbox o{\x} }
\foreach \x in {57,61,62,63} { \sgbox r{\x} }
\begin{scope}[xshift=16cm]
\draw[sgline] (-.55,.55) rectangle ++(14.1,-5.1);
\def\sgm{14}
\def\sgc{50}
\sgbox00
\foreach \x in {17,19} { \sgbox1{\x} }
\foreach \x in {34,36,38} { \sgbox2{\x} }
\foreach \x in {51,53,55,57} { \sgbox3{\x} }
\foreach \x in {14,28,42,56,31,33,45,47,59,61,63,48,50,52,62} { \sgbox o{\x} }
\foreach \x in {54,58,60,63} { \sgbox r{\x} }
\end{scope}
\end{tikzpicture}%
\caption[]{Two numerical semigroups. The top row represents the threshold
interval, while the row below, numbering~$q$, represent left elements. The
position of~$0$ in the bottom row is~$\rho$. Elements in the sumsets~$i\Gamma$
are put in different colors. The elements that they determine are in black in
the column above them. Right generators are in grey.

Left: the semigroup~$\langle14,22,23\rangle_{56}$, the smallest
Eliahou-negative semigroup with~$k = 13$, $\ell = 3$, $r = 4$, $q = 4$, $\rho
= 0$, $s = 10$, so~$E = -1$ and~$W = 35$.

Right: the semigroup~$\langle14,17,19\rangle_{50}$ has~$k = 14$, $\rho =
6$ so~$E = 8$ and~$W = 48$.}
\label{fig:full}
\end{figure}

\subsection{Two cases of Eliahou-positivity}

We give below two simple necessary conditions for Eliahou-negativity, that to
our knowledge do not appear in the literature.

\begin{theorem} \label{thm:h3}
If~$3\gamma_1\ge c+m$, then~$S$ is Eliahou-positive.
\end{theorem}

\begin{proof}
The left elements and threshold non-generators of~$S$ are determined by the
elements of~$\{0\}$, $\Gamma$ and~$2\Gamma$. As the elements of~$\Gamma$ are
generators, they cannot determine the same left elements as each other or~$0$.

By writing~$\ceil{x}\ge x$ and~$c = qm-\rho$, and since any integer greater
than~$q-1$ is at least~$q$, we have the inequality:
\[\ceil{\frac{\mathstrut c-\gamma_i}{m}} + \ceil{\frac{\mathstrut c-\gamma_j}{m}} \ge \begin{cases}
q&\text{if~$\gamma_i+\gamma_j < c + m - \rho$}\\
q-1&\text{if~$c+m-\rho\le\gamma_i+\gamma_j < c+m$}\\
0&\text{otherwise.}
\end{cases}
\]
Summing the inequality over all~$1\le i\le j\le\ell-1$, we get on the left
hand side~$\ell$ times~$\smash{\ceil{\frac{c-\gamma_i}{m}}}$ for all~$i$ (twice
for~$i = j$ and~$\ell-2$ times for~$i\ne j$), which is at most~$(k-q)\ell$
since there are~$q$ left elements determined by~$0$. On the right hand side,
as at most~$\rho$ elements can be between~$c+m-\rho$ and~$c+m$ and as there
are~$\ell$ threshold non-generators determined by~$0$ and elements
of~$\Gamma$, we have at least~$q(s-\ell)-\rho$. This shows that~$E\ge0$.
\end{proof}

\begin{theorem} \label{thm:l3}
If~$\ell\le2$, then~$S$ is Eliahou-positive.
\end{theorem}

\begin{proof}
If~$\ell = 1$, then~$k = q$ and~$s = 1$, so the result is clear (and
well-known). If~$\ell = 2$, left elements and threshold non-generators are
determined by~$i\gamma_1$ for~$i = 0,\dotsc,s-1$, with~$(s-1)\gamma_1 < c+m$
(either~$s\gamma_1\ge c+m$ or~$m$ divides~$s\gamma_1$).

By writing~$\ceil{x} = x + \fpart{-x}$ and~$c = qm-\rho$, we get again the
inequality:
\[\ceil{\frac{\mathstrut c-i\gamma_1}{m}} + \ceil{\frac{\mathstrut
c-(s-1-i)\gamma_1}{m}} \ge \begin{cases}
q&\text{if~$(i\gamma_1-c\bmod m) \ge \rho$}\\
q-1&\text{otherwise.}
\end{cases}\]
Summing this for~$i = 0,\dotsc,s-1$, we get~$2k \ge qs - \rho$ since
only~$\rho$ elements~$i\gamma_1-c$ can have a congruence modulo~$m$ less
than~$\rho$.
\end{proof}

\section{Collision-free \texorpdfstring{$h$}{h}-regular semigroups}
\label{sec:regular}

Here, we study the two classes to which all the
Eliahou-negative semigroups in our experimental set either belong or nearly
belong (see Section~\ref{sec:negative} for examples).
We give some of their properties and show ways to construct such semigroups.

\subsection{Definitions} \label{sec:definitions}

\begin{definition} \label{def:regular}
If~$h\ge2$ is an integer, we say that~$S$ is \emph{\h-regular}
if~$h\Gamma\subset[c,c+m)$.
\end{definition}

\begin{definition} \label{def:collision}
We say that~$S$ is \emph{collision-free} if every element less than~$c+m$ can
be written in a unique way as a combination of generators.
\end{definition}

For reasons that will be apparent shortly, we introduce the \h-Farey interval
containing the number~$\sfrac{c+m}{hm}$, that is, the shortest
interval~$(a'/b',a/b]$ such that:
\[\frac{a'}{b'} < \frac{c+m}{hm} \le \frac ab\]
where~$a'/b'$ and~$a/b$ are reduced fractions with~$b\le h$ and~$b'\le h$.
As the interval $(ia'/b',ia/b)$ does not contain an integer, we have for~$i =
1,\dotsc,h$:
\begin{equation} \label{farey}
\ceil{\frac{i(c+m)}{hm}} - 1 = \ceil{\frac{\mathstrut ia}{b}} - 1 =
\floor{\frac{ia'}{b'}}\text.
\end{equation}
We also use the notations~$\phi'_i = \fpart{ia'/b'}$ and~$\phi_i =
\fpart{-ia/b'}$, the distances between the interval~$(ia'/b',ia/b]$ and the
previous and the next integer, respectively.

For example, both semigroups in Figure~\ref{fig:full} are
$3$\nobreakdashes-regular and collision-free with Farey interval~$(3/2,5/3]$.

\subsection{Number of left elements}

If~$S$ is \h-regular, then all left elements and threshold non-generators are
determined by elements of~$i\Gamma$ for~$i = 0,\dotsc,h$. In this case, $S$ is
collision-free if and only if~$\Gamma$ is a~$B_h$~set and the
sumsets~$i\Gamma$ for~$i = 0,\dotsc,h$ are pairwise disjoint modulo~$m$. If
this is true, then these elements, along with the right generators, form the
\emph{Apéry set} of~$S$ with respect to~$m$.

Let~$\lambda\in i\Gamma$. As~$h\Gamma\subset[c,c+m)$, the number of left
elements determined by~$\lambda$ satisfies:
\begin{equation} \label{excess}
\ceil{\frac{\mathstrut c-\lambda}{m}} = 
\ceil{\frac{(h-i)(c+m)}{hm}} - 1 +
x(\lambda)\quad\text{where}\quad x(\lambda)\in\{0,1\}\text.
\end{equation}
We call \emph{excess} of~$\lambda$ the number~$x(\lambda)$; the \emph{excess}
of the semigroup~$S$, denoted by~$x$, is the sum of all~$x(\lambda)$
in~$i\Gamma$ for~$i = 0,\dotsc,h-1$.
For example, the semigroup on the left of
Figure~\ref{fig:full} has excess zero; the semigroup on the right has~$x(34)
= 1$, as~$34$ determines one more left element than~$36$ and~$38$.

The formula~\eqref{excess} is the reason why we considered the \h-Farey
interval of~$\sfrac{c+m}{hm}$: using~\eqref{farey}, the rank~$k$ of the
semigroup (as well as~$q = \ceil{c/m}$) depends only on~$h$, $a/b$, $\ell$
and~$x$, not on~$m$, $\Gamma$ or~$c$. These numbers thus capture the
combinatorics of the semigroup. 
This is expanded upon in Sections~\ref{sec:numbers} and~\ref{sec:negative}.

Still for~$\lambda\in i\Gamma$, define the number~$\theta(\lambda)$ as:
\begin{equation} \label{th}
\theta(\lambda) = \lambda + \floor{\frac{(h-i)a'}{b'}}m - c\text.
\end{equation}
As~$\lambda \le i(c+m)/h$, we have~$\theta(\lambda) < m$.
From~\eqref{excess} and~\eqref{farey}, we have~$x(\lambda) = 0$ if and only
if~$\theta(\lambda)\ge 0$. Moreover, as~$\theta(\lambda)\equiv\lambda-c$
modulo~$m$, this number shows where the threshold element determined
by~$\lambda$ falls into the threshold interval.

As we want to minimize~$k$ to find Eliahou-negative semigroups, it is natural
to minimize~$x$, which motivates the following result.

\begin{proposition} \label{prop:excess}
We have~$x = 0$ if and only if~$x(i\gamma_1) = 0$, where:
\[i = \begin{cases}
h-b&\text{if~$\gamma_1\ge a'm/b'$,}\\
h-(b\bmod b')&\text{otherwise.}
\end{cases}\]
If~$x = 0$ and~$\gamma_{\ell-1} > a'm/b'$, then~$S$ is canonical.
\end{proposition}

The semigroups in Figure~\ref{fig:full} illustrate these two cases with
the interval~$(3/2,5/3]$: on the
left, we have~$\gamma_1\ge a'm/b'$ and the leftmost element in the threshold
interval is determined by~$0 = (h-b)\gamma_1$. On the right, we have~$\gamma_1
< a'm/b'$ and the only element with excess~$1$ is~$34 = [h-(b\bmod
b')]\gamma_1$.

\begin{proof}
As~$i\gamma_1$ is the minimum value of~$i\Gamma$, we have~$x = 0$ if
and only if~$\theta(i\gamma_1)\ge0$ for $i = 0,\dotsc,h-1$. Write~$i
= h - ub + u'b'$ with~$1\le u\le b'$. As~$ab' = a'b + 1$, we get:
\begin{equation} \label{uu'}
\theta(i\gamma_1) = h\gamma_1 - c - m + (am-b\gamma_1)u
+ (b'\gamma_1-a'm)u'\text.
\end{equation}
If~$\gamma_1\ge a'm/b'$, then as~$u\ge1$, we
have~$\theta(i\gamma_1) \ge \theta((h-b)\gamma_1)$.
Therefore, we have~$x = 0$ if~$x((h-b)\gamma_1) = 0$.

If~$\gamma_1\le a'm/b'$, let~$a''/b''$ be the fraction:
\[\frac{a''}{b''} = \frac{(1+\floor{b/b'})a' - a}{(1 + \floor{b/b'})b' -
b}\text.\]
Also let~$u'' = (1+\floor{b/b'})u - u'$, so~$i = h - u''b' + ub''$. This gives:
\[\theta(i\gamma_1) = h\gamma_1 - c - m + (b''\gamma_1 - a''m)u + (a'm
-b'\gamma_1)u''\text.\]
Applying this identity to~$i = h - b' + b'' = h - (b\bmod b')$, so~$u'' = u' =
1$, and as~$\theta((h-b')\gamma_1) = (h-b')\gamma_1 + a'm - c$ from the
definition~\eqref{th}, we get:
\[b''\gamma_1 - a''m = \theta\bigl((h-b'+b'')\gamma_1\bigr) + m -
\theta\bigl((h-b')\gamma_1\bigr)\text,\]
which is positive if~$\theta((h-b'+b'')\gamma_1)\ge0$. Therefore,
as~$u\ge1$ and so~$u''\ge1$ since~$i < h$, we
have~$\theta(i\gamma_1)\ge\theta((h-b'+b'')\gamma_1)$ and
thus~$x = 0$.

Finally, if~$\lambda\in i\Gamma$, then~$\lambda\le i\gamma_{\ell-1}$. The
condition~$\gamma_{\ell-1} > a'm/b'$ implies that~$\floor{(h-i)a'/b'} <
(h-i)\gamma_{\ell-1}$. As~$h\gamma_{\ell-1}\le c+m-1$, we get~$\theta(\lambda)
< m-1$. If~$x = 0$, then we have~$\theta(\lambda)\ge0$ and
thus~$\lambda\nequiv c-1$ modulo~$m$.  This shows that~$c-1\not\in S$ and
so~$S$ is canonical.
\end{proof}

\subsection{Split semigroups}

The results of this paper rely on the condition that~$S$ is collision-free,
which is difficult to check. We give here a sufficient condition
that is more tractable.

\begin{definition} \label{def:split}
We say that $S$ is \emph{split} if $\Gamma/m \subset (a'/b',a/b)$ and the
sumsets~$i\Gamma$ for~$i = 0,\dotsc,h$ lie in disjoint intervals modulo~$m$.
\end{definition}

For example, the semigroup on the left of Figure~\ref{fig:full} is split, but
the one on the right is not.

If the intervals containing~$i\Gamma$ are disjoint modulo~$m$, then the
set~$\Gamma/m$ has to be contained in an \h-Farey interval: if it
straddles~$a'/b'$ (or any other \h-Farey fraction), then~$b'\Gamma$
straddles~$0$ modulo~$m$. The first condition of the definition is thus that
this interval is the same as that of the semigroup.

Let~$d = \gamma_{\ell-1} - \gamma_1$ be the diameter of~$\Gamma$.

\begin{proposition} \label{prop:split}
The semigroup~$S$ is split if and only if:
\[\Gamma \subset \left[\frac{a'm+(h-b')d+1}{b'},\frac{am-(h-b)d-1}{b}\right]\text.\]
If, furthermore, $\Gamma$ is a~$B_h$~set, then~$S$ is collision-free.
\end{proposition}

\begin{proof}
We may assume that~$\gamma_1 > a'm/b'$, which is implied by both assertions.
Consider the numbers~$\kappa_j = [h-(jb\bmod b+b')]\gamma_1$ for~$j =
1,\dotsc,b+b'$. Some of these elements may be negative, but they cover
all~$i\gamma_1$ for~$i = 0,\dotsc,h$, since~$b+b'\ge h$ and~$b$ and~$b'$ are
coprime. As~$\kappa_{j+1} - \kappa_j$ is either~$-b\gamma_1$ or~$b'\gamma_1$,
the formula~\eqref{uu'} in the proof of Proposition~\ref{prop:excess} shows
that~$\theta(\kappa_j) < \theta(\kappa_{j+1})$.  As the last
element~$\theta(\kappa_{b+b'})$ is~$\theta(h\gamma_1) = h\gamma_1 -
c$, we have~$\theta(\kappa_{b+b'}) < \theta(\kappa_1) + m$, so the
numbers~$\kappa_i$ are cyclically ordered modulo~$m$. This shows that~$S$ is
split if and only if:
\begin{itemize}
\item $\theta(h\gamma_{\ell-1}) < \theta((h-b)\gamma_1)$, equivalent
to~$hd + 1 \le am - b\gamma_1$;
\item $\theta((h-b')\gamma_{\ell-1}) < \theta(h\gamma_1) + m$,
equivalent to~$(h-b')d + 1 \le b'\gamma_1 - a'm$.
\end{itemize}
This is in turn equivalent to~$\Gamma$ being in the announced interval.
Moreover, if the sets~$i\Gamma$ are in disjoint intervals modulo~$m$, they are
disjoint modulo~$m$ and thus~$S$ is collision-free if~$\Gamma$ is a~$B_h$~set.
\end{proof}

\subsection{The semigroup \texorpdfstring{$\hat S(h,a/b,\Delta)$}{Ŝ(h, a/b,
Δ)}} \label{sec:Shat}

With the criteria above, we can now explicitly construct a split semigroup
with excess zero for any~$h$, $a/b$ and~$\ell$, which is in a sense the
smallest.

\begin{definition} \label{def:Shat}
Let~$h\ge2$ be an integer, $a/b$ be an \h-Farey fraction and~$a'/b'$ its
predecessor and $\Delta\subset\mathbb N$ be a finite~$B_h$~set containing~$0$.
Let~$d = \max\Delta$. Define the semigroup~$\hat S(h, a/b, \Delta) =
\langle\hat m,\hat\Gamma\rangle_{\hat c}$, where:
\begin{align*}
\hat m &= (b+b')(hd+1) - bb'd\text,\\
\hat\Gamma &= (a+a')(hd+1) - ab'd + \Delta\text,\\
\hat c &= h\max\hat\Gamma + 1 - \hat m\text.
\end{align*}
\end{definition}

Define the following parameter of the semigroup~$S$:
\[\tau = \frac{ham}{b} - m - c\text.\]
Thus, we have~$\rho = \phi_hm + \tau$. Recalling that~$ab' = a'b + 1$, we
compute with obvious notation:
\[\hat\tau = \frac{(h-b)(hd+1)}{b}\text.\]

\begin{proposition} \label{prop:Shat}
The semigroup~$\hat S(h,a/b,\Delta)$ is canonical, split and has~$x = 0$.
Assume that~$S$ and~$\hat S$ have the same~$d$. If~$S$ is split,
then~$m\ge\hat m$. If~$x = 0$, then~$\tau\ge\hat\tau$.
\end{proposition}

Unfortunately, the minimum diameter of a~$B_h$~set of given cardinality is not
known (see \cite{kovacevic,obryant} for bounds). A simple choice is~$\Delta =
\{0,1,h+1,h^2+h+1,\dotsc\}$, but tighter sets exist when~$\ell\ge5$.

\begin{proof}
As~$ab' = a'b + 1$, the set~$\Gamma$ just fits into the interval in
Proposition~\ref{prop:split} and we have~$\theta((h-b)\gamma_1) = 0$,
so~$S$ is split, has excess zero and is canonical from
Proposition~\ref{prop:excess}.

Conversely, if~$S$ is split, then the width of the interval in
Proposition~\ref{prop:split} is at least~$d$,
so~$m\ge\hat m$.  Combining~$\theta((h-b)\gamma_1) = (h-b)\gamma_1 +
am - m - c\ge 0$ and~$h\gamma_{\ell-1}+1 = h\gamma_1 + hd + 1\le c+m$, we
get~$\tau\ge\hat\tau$.
\end{proof}

\subsection{Inflating \texorpdfstring{$h$}{h}-regular semigroups}

We conclude this section by giving two ways of constructing \h-regular
semigroups starting from~$S$, illustrated in Figure~\ref{fig:+*}. Define:
\begin{align*}
S^\uparrow &= \langle m,\Gamma + m\rangle_{c+hm}\text,\\
S^\rightarrow &= \langle m+b,\Gamma+a\rangle_{c+ha-b}\text.
\end{align*}
These constructions are similar to the ones found in~\cite{delgado}. Other
ways to enlarge semigroups exist in the literature~\cite{rosales,barucci}, but
they are different.

\begin{figure}[ht]%
\begin{tikzpicture}[box semigroups]
\draw[sgline] (-.55,.55) rectangle ++(14.1,-5.1);
\def\sgm{14}
\def\sgc{56}
\sgbox00
\foreach \x in {22,23} { \sgbox1{\x} }
\foreach \x in {44,45,46} { \sgbox2{\x} }
\foreach \x in {66,67,68,69} { \sgbox3{\x} }
\foreach \x in {14,28,42,56,36,37,50,51,64,65,58,59,60} { \sgbox o{\x} }
\foreach \x in {57,61,62,63} { \sgbox r{\x} }
\begin{scope}[xshift=16cm,yshift=1.5cm]
\draw[sgline] (-.55,.55) rectangle ++(14.1,-8.1);
\def\sgm{14}
\def\sgc{98}
\sgbox00
\foreach \x in {36,37} { \sgbox1{\x} }
\foreach \x in {72,73,74} { \sgbox2{\x} }
\foreach \x in {108,109,110,111} { \sgbox3{\x} }
\foreach \x in {14,28,42,56,70,84,98,50,51,64,65,78,79,92,93,106,107,86,87,88,100,101,102} { \sgbox o{\x} }
\foreach \x in {99,103,104,105} { \sgbox r{\x} }
\end{scope}
\begin{scope}[xshift=6.5cm,yshift=-8.5cm]
\draw[sgline] (-.55,.55) rectangle ++(17.1,-5.1);
\def\sgm{17}
\def\sgc{68}
\sgbox00
\foreach \x in {27,28} { \sgbox1{\x} }
\foreach \x in {54,55,56} { \sgbox2{\x} }
\foreach \x in {81,82,83,84} { \sgbox3{\x} }
\foreach \x in {17,34,51,68,44,45,61,62,78,79,71,72,73} { \sgbox o{\x} }
\foreach \x in {69,70,74,75,76,77,80} { \sgbox r{\x} }
\end{scope}
\end{tikzpicture}%
\caption[]{Top left: the $3$\nobreakdashes-regular
semigroup~$S = \langle14,22,23\rangle_{56} = \hat S(3,5/3,3)$.

Right: the semigroup~$S^\uparrow = \langle14,36,37\rangle_{98}$.

Bottom: the semigroup~$S^\rightarrow = \langle17,27,28\rangle_{68}$.}
\label{fig:+*}
\end{figure}

It is easily checked that both of these semigroups are \h-regular.
Let~$a^\uparrow/b^\uparrow$, etc., be the parameters of~$S^\uparrow$ and similarly for~$S^\rightarrow$.
Let~$S^{\rightarrow n}$ be the semigroup obtained by applying~$n$ times the
transformation~$S\mapsto S^\rightarrow$.

\begin{proposition} \label{prop:S+}
We have~$a^\uparrow/b^\uparrow = a/b + 1$, $\rho^\uparrow = \rho$ and~$x^\uparrow = x$. The
semigroup~$S^\uparrow$ is split if and only if~$S$ is split and collision-free if and
only if~$S$ is collision-free.
\end{proposition}

\begin{proof}
The values of~$a^\uparrow/b^\uparrow$ and~$\rho^\uparrow$ are easily computed from the
definitions. Proposition~\ref{prop:split} shows that $S^\uparrow$ is split if and
only if~$S$ is split.

To compute~$x^\uparrow$, let~$\lambda\in i\Gamma$ and let~$\lambda^\uparrow = \lambda +
im$ be the corresponding element of~$i\Gamma^\uparrow$. We
have~$\theta(\lambda^\uparrow) = \theta(\lambda)$ from~\eqref{th}, which shows
that~$x(\lambda^\uparrow) = x(\lambda)$ and~$x^\uparrow = x$.

Finally, $S$ is collision-free if and only if~$\Gamma$ is a~$B_h$~set and the
sets~$i\Gamma$ are pairwise disjoint modulo~$m$. As~$\Gamma^\uparrow\equiv\Gamma$
modulo~$m$, $S^\uparrow$ is collision-free if and only if~$S$ is collision-free.
\end{proof}

\begin{proposition} \label{prop:S*}
We have~$a^\rightarrow/b^\rightarrow = a/b$, $\tau^\rightarrow = \tau$ and~$x^\rightarrow\le x$. If~$S$
is split, then~$S^\rightarrow$ is split. If~$S$ is collision-free, then~$S^{\rightarrow n}$ is
collision-free if~$n$ is large enough.
\end{proposition}

\begin{proof}
As above, the values of~$a^\rightarrow/b^\rightarrow$ and~$\tau^\rightarrow$ are easily computed from
the definitions. The upper end of the interval in Proposition~\ref{prop:split}
is advanced by~$a$ and the lower end by~$a - 1/b'$, which shows
that~$S^\rightarrow$ is split if~$S$ is split.

Let~$\lambda\in i\Gamma$ and let~$\lambda^\rightarrow = \lambda + ia$ be the
corresponding element in~$i\Gamma^\rightarrow$. From~\eqref{th} and~\eqref{farey}, we
have~$\theta(\lambda^\rightarrow) = \theta(\lambda) + \phi_{h-i}b$. This shows
that~$x(\lambda^\rightarrow) \le x(\lambda)$ and thus~$x^\rightarrow\le x$.

Finally, let~$\lambda\in i\Gamma$ and~$\mu\in j\Gamma$. We have
either~$\phi_{h-i} = \phi_{h-j}$ (if~$i\equiv j$ modulo~$b$),
so~$\theta(\lambda^{\rightarrow n}) - \theta(\mu^{\rightarrow n}) =
\theta(\lambda) - \theta(\mu)$, or~$\phi_{h-i}\ne\phi_{h-j}$,
so~$\lambda^{\rightarrow n}$ and~$\mu^{\rightarrow n}$ are far apart
modulo~$m^{\rightarrow n}$ if~$n$ is large enough. This shows
that~$S^{\rightarrow n}$ is collision-free if~$S$ is collision-free.
\end{proof}

\section{Eliahou and Wilf numbers} \label{sec:numbers}

In this section, we show how to compute the Eliahou number of a collision-free
\h-regular semigroup and show that this class does not contain counterexamples
to Wilf's conjecture. We also give a partial result on semigroups with
collisions.

\subsection{Eliahou number}

If~$S$ is collision-free \h-regular, the number of elements in the
sumset~$i\Gamma$ is:
\[s_i = \mchoose{\ell-1}{i}\text.\]
The number~$s$ of threshold non-generators is:
\[s = s_0 + \dotsb + s_h = \mchoose{\ell}{h}\text.\]
In the following, we make use of the following two summation identities:
\begin{align}
\label{chu} &\ell\sum_{i=0}^{h-1}\frac{(h-i)s_i}{h} = s\text,\\
\label{isi} &\ell\sum_{i=0}^{h-1}\frac{is_i}{h} = (h-1)s_h\text.
\end{align}
The first is a consequence of the multiset version of the Chu-Vandermonde
identity and~$\ell\smchoose{\ell+1}{h-1} = hs$, while the second is derived
by writing~$is_i = (\ell-1)\smchoose{\ell}{i-1}$ and, after summation,
$\ell(\ell-1)\smchoose{\ell+1}{h-2} = h(h-1)s_h$.

\begin{theorem} \label{thm:E}
If~$S$ is collision-free \h-regular, its
Eliahou number is given by~$E = E_0(h,a/b,\ell) + \ell x + \rho$,
where~$E_0(h,a/b,\ell)$ has the equivalent forms:
\begin{align}
\label{E'} E_0(h,a/b,\ell) &= -\ell\sum_{i=0}^{h-1}\phi'_{h-i}s_i + \phi'_hs\text,\\
\label{E} E_0(h,a/b,\ell) &= -(h-1)s_h + \ell\sum_{i=0}^{h-1}\phi_{h-i}s_i - \phi_hs\text,\\
\label{E'a} E_0(h,a/b,\ell) &= -\floor{h\phi'_1}s + \ell\sum_{j=1}^{\floor{h\phi'_1}}
\mchoose{\ell}{h-\ceil{j/\phi'_1}}\text,\\
\label{Ea} E_0(h,a/b,\ell) &= -(h-1)s_h + \floor{h\phi_1}s - \ell\sum_{j=1}^{\floor{h\phi_1}}
\mchoose{\ell}{h - \ceil{j/\phi_1}}\text.
\end{align}
\end{theorem}

\begin{proof}
We use~$E = k\ell - qs + \rho$.
From~\eqref{excess} and~\eqref{farey}, we have:
\begin{equation*} \label{k}
k = \sum_{i=0}^{h-1}\left(\frac{(h-i)a'}{b'} - \phi'_{h-i}\right)s_i + x\text.
\end{equation*}
Also from~\eqref{farey}, we have~$q = ha'/b' - \phi'_h$, which
yields~\eqref{E'} from~\eqref{chu}.

Writing instead~$k$ and~$q$ in terms of~$a/b$ and~$\phi_i$, and using the sum
of~\eqref{chu} and~\eqref{isi} to compute~$\ell\mkern2mu\smash{\sum}s_i$, we
get~\eqref{E}.

To show~\eqref{E'a}, write~$\phi'_{h-i} = (h-i)\phi'_1 -
\floor{(h-i)\phi'_1}$. We use~\eqref{chu} as well as Abel's lemma, or
summation by parts. This states that:
\[\sum_{i=0}^{h-1} u_i(v_{i+1} - v_{i}) = u_{h}v_{h} - u_{0}v_{0} -
\sum_{i=0}^{h-1} (u_{i+1} - u_{i})v_{i+1}\text.\]
We use~$u_i = \floor{(h-i)\phi'_1}$ and~$v_i = \smchoose{\ell}{i-1}$, so~$s_i
= v_{i+1} - v_{i}$. Moreover, we have~$u_{h} = v_0 = 0$ and~$u_{i+1} - u_{i} =
-1$ if and only if~$h-i = \ceil{u_i/\phi'_1}$, which shows~\eqref{E'a}. The
case of \eqref{Ea} is identical.
\end{proof}

As seen in the following result, the value~$E_0(h,a/b,\ell)$ is negative in
most cases. How this translates to a negative~$E$ is the focus of
Section~\ref{sec:negative}.

\begin{theorem} \label{thm:E0<0}
If~$\ell\ge3$ and~$\fpart{a/b}\ne1/h$, then~$E_0(h,a/b,\ell) < 0$.
\end{theorem}

\begin{proof}
According to~\eqref{Ea}, if~$\floor{h\phi_1}$ is kept constant,
$E_0(h,a/b,\ell)$ is decreasing in~$\phi_1$. Therefore, we may assume
that~$\phi_1$ is a multiple of~$1/h$, that is, $b$ divides~$h$.

In this case, the numbers~$\phi_{h-i}$ for~$i = 0,\dotsc,h-1$ consist of~$i/b$
for~$i = 0,\dotsc,b-1$, each appearing~$h/b$ times and are thus bounded by a
permutation of~$(i/h)$. This permutation is the identity if and only
if~$\phi_1 = 1 - 1/h$, or~$\fpart{a/b} = 1/h$.
Therefore, by the rearrangement inequality, the sum of~$\phi_{h-i}s_i$ is
bounded by the sum of~$is_i/h$, which is given by~\eqref{isi} and we
have~$E_0(h,a/b,\ell) < 0$ by~\eqref{E}, since we have~$\phi_h = 0$.
\end{proof}

\subsection{Wilf number} \label{sec:wilf}

\begin{theorem} \label{thm:W}
Every collision-free \h-regular semigroup is Wilf-positive.
\end{theorem}

As~$W = E + (k-q)r$ and~$E$ can be negative, proving such a result has to
involve a lower bound on~$r$. We use the one below, based on a reference that
was pointed to us by Johan Mercedes.

\begin{lemma} \label{lem:sidon}
Every collision-free \h-regular semigroup satisfies:
\[r\ge\binom{\ell-1}{2}\floor{\frac{\mathstrut h}2}\ceil{\frac{\mathstrut h}2}\text.\]
\end{lemma}

\begin{proof}
It is shown in~\cite{kovacevic} that if~$\Gamma$ is a~$B_h$~set modulo~$m$,
then~$m$ is at least the number of $(\ell-1)$\nobreakdashes-dimensional vectors
of integers such that the sum of the positive coordinates is at
most~$\floor{h/2}$ and the sum of the absolute value of the negative
coordinates is at most~$\ceil{h/2}$. As there are~$\sbinom{n}{i}$ vectors
of~$i$ positive integers summing to at most~$n$, summing over the sets of
positive and negative coordinates, we get:
\[m\ge \sum_{\substack{i,j\ge0\\ i+j \le \ell-1}} \binom{\ell-1}{i+j}
\binom{i+j}{i}
\binom{\floor{h/2}}{i} \binom{\ceil{h/2}}{j}\text.\]
Writing~$\sbinom{i+j}{i} = 1 + \left[\binom{i+j}{i}-1\right]$ and using the
Chu-Vandermonde identity twice yields:
\[m\ge \binom{\ell+h-1}{h} +
\sum_{\substack{i,j\ge0\\i+j\le\ell-1}}\binom{\ell-1}{i+j}\left[\binom{i+j}{i}
- 1\right]
\binom{\floor{h/2}}{i} \binom{\ceil{h/2}}{j}\text.\]
As~$s = \sbinom{\ell+h-1}{h}$ and~$m = r + s$, the summand~$i = j = 1$ gives
the result.
\end{proof}

\begin{proof}[Proof of Theorem~\ref{thm:W}]
Theorems~\ref{thm:h3} and~\ref{thm:l3} show that if~$h\le2$ or~$\ell\le2$,
then~$S$ is Eliahou-positive and thus Wilf-positive. If~$h\ge3$
and~$\ell\ge3$, we have from~\eqref{W}, \eqref{Er} and~\eqref{Es}:
\[W = E - qr + \frac{(E + qs - \rho)r}{\ell}\text.\]
Bounding every summand in~\eqref{E'a} from below by~$1$, we get~$E\ge
-\floor{h\phi'_1}(s-\ell) + \rho$. This shows:
\[W\ge\frac{s-\ell}{\ell}\bigl[qr - \floor{h\phi'_1}(\ell+r)\bigr]\text.\]
Finally, we have~$q = \floor{ha'/b'} \ge h + \floor{h\phi'_1}$
and~$\floor{h\phi'_1}\le h-1$. Thus, we get:
\[W\ge\frac{s-\ell}{\ell}\bigl[h(r-\ell) + \ell\bigr]\text.\]
With the bound of Lemma~\ref{lem:sidon}, the quantity~$h(r-\ell)+\ell$ is
always non-negative for~$\ell\ge3$ and~$h\ge3$, which gives the result.
\end{proof}

\subsection{Semigroups with collisions} \label{sec:collisions}

The results in the section apply to collision-free semigroups. Obviously, if
collisions occur in the threshold interval, the semigroup has more right
generators than it would in the collision-free case and so higher values
of~$E$ and~$W$. If collisions occur in the left elements, however, the
semigroup has a lower rank; one might think that it could lead to a lower~$E$
or~$W$. The result below, while partial, suggests that this is not the case.

In the following, a \emph{collision} is a pair of sums of at most~$h$ elements
of~$\Gamma$ that give the same result modulo~$m$. A collision at~$\lambda$ is
\emph{primitive} if one cannot write~$\lambda = \mu + \nu$ where a
smaller collision occurs at~$\mu$.

\begin{theorem} \label{thm:collisions}
If $S$ has one primitive collision, then~$E > E_0(h,a/b,\ell) + \rho + \ell x$.
\end{theorem}

\begin{proof}
Assume that the primitive collision occurs at~$\lambda$. Consider the sum
giving~$\lambda$ with the higher number of generators, say~$h-j$. Any sum of
generators containing that sum is redundant. There are~$s_i$ such sums in the
sumset~$(h-j+i)\Gamma$ for~$i = 0,\dotsc,j-1$. The decrease in left elements
compared to the collision-free case is thus, at most:
\[\sum_{i=0}^{j-1}\left(\floor{\frac{(j-i)a'}{b'}}+1\right)s_i\text.\]
The decrease in threshold non-generators is $s_0 + \dotsb + s_j =
\smchoose{\ell}{j}$. Therefore, using~$q = ha'/b' - \phi'_h$ and~\eqref{chu}, the change
to the Eliahou number~$k\ell-qs+\rho$ is at least:
\[\left(\frac{(h-j)a'}{b'} - \phi'_h\right)\mchoose{\ell}{j} -
\mchoose{\ell}{j-1} +
\ell\sum_{i=0}^{j-1}\phi'_{j-i}s_i\text.\]
Since~$a'/b' \ge 1$ and $j \le h-2$ ($\Gamma$ does not collide with itself),
this is positive.
\end{proof}

\section{Families of Eliahou-negative semigroups} \label{sec:negative}

The goal of this section is to discuss for which values of~$h$, $a/b$
and~$\ell$ there exist Eliahou-negative collision-free \h-regular semigroups
and how to construct them. Theorems~\ref{thm:h3} and~\ref{thm:l3} show that
necessary conditions are~$h\ge3$ and~$\ell\ge3$. Moreover, the
number~$E_0(h,a/b,\ell)$ appearing in Theorem~\ref{thm:E} only depends on the
fractional part of~$a/b$. With Proposition~\ref{prop:S+}, this shows
that the semigroup~$S^\uparrow$ has the same Eliahou number as~$S$ if~$S$ is
collision-free (this is not true otherwise: the
semigroup~$\langle8,11,13\rangle_{32}$ has~$E = 11$, but~$S^\uparrow =
\langle8,19,21\rangle_{56}$ has~$E^\uparrow = 17$).

Theorem~\ref{thm:E0<0} shows that another necessary condition
is~$\fpart{a/b}\ne1/h$, in other words, that the Farey interval is not
directly after an integer. Furthermore, this result shows that we need to find
semigroups with low enough~$x$ and~$\rho$ so that~$E_0(h,a/b,\ell) + \ell x +
\rho < 0$.

It turns out that there are three regimes depending on the value of~$b$
relative to~$h$, which we present by increasing order of complexity. We then
briefly discuss semigroups that are not collision-free \h-regular. We give
some examples, many of which are taken from the experimental set that we
obtain in the next section. We also show how the Eliahou-negative semigroups
already published fit into this classification.

\subsection{The case \texorpdfstring{$b = h$}{b = h}}

The case~$b = h$ is simplest. Consider any semigroup of the form~$\langle
m,\Gamma\rangle_{(a-1)m}$, which has~$\rho = 0$. If~$\Gamma$ is a~$B_h$~set
contained in the interval in Proposition~\ref{prop:split} (this always exists
when~$m$ is large enough), then~$S$ is split and thus collision-free.
Proposition~\ref{prop:excess} shows that~$x = 0$ as~\eqref{excess} shows that
we always have~$x(0) = 0$. This shows that~$E < 0$ under the conditions of
Theorem~\ref{thm:E0<0}.

This includes Delgado's semigroups~\cite{delgado}, which are equal to~$\hat
S\sparen{h,\frac{2h-1}{h},\{0,1\}}$ and its iterated images by~$S\mapsto
S^\uparrow$ and~$S\mapsto S^\rightarrow$, and Eliahou and Fromentin's
semigroups~\cite{eliahou-fromentin}, which have~$h = 3$ and~$a/b = 5/3$.
When~$a = 2h-1$, we have~$\phi_1 = 1/h$ and thus~\eqref{Ea} shows that:
\[E_0\left(h,\frac{2h-1}{h},\ell\right) = -(h-1)\mchoose{\ell-1}{h} + \mchoose{\ell}{h} - \ell\text.\]
This formula specializes to~$-\sbinom{h-1}{2}$ when~$\ell = 3$ (Delgado)
and~$-\sbinom{\ell}{3}$ when~$h = 3$ (Eliahou--Fromentin). However, using
any value~$a\not\equiv1\bmod h$ coprime with~$h$, we can get a much
greater variety of semigroups with~$E < 0$ (Figure~\ref{fig:b=h}).

\begin{figure}[ht]%
\begin{tikzpicture}[semigroups]
\sgdraw{50,68,69}{300}
\begin{scope}[yshift=-5cm]
\sgdraw{50,78,79}{350}
\end{scope}
\begin{scope}[yshift=-10cm]
\sgdraw{50,88,89}{400}
\end{scope}
\sglegend{5}
\end{tikzpicture}
\caption[]{%
Figures in this section follow a more compact convention: we only show the
threshold interval of the semigroup, with elements determined by the
sumsets~$i\Gamma$ in different colors and symbols. Gaps represent right
generators.

Top: the semigroup~$\langle50,68,69\rangle_{300}$ ($a/b = 7/5$, $E = -3$).

Middle: the semigroup~$\langle50,78,79\rangle_{350}$ ($a/b = 8/5$, $E = -3$).

Bottom: the semigroup~$\langle50,88,89\rangle_{400}$ ($a/b = 9/5$, $E = -6$).}
\label{fig:b=h}
\end{figure}

\subsection{The case \texorpdfstring{$b\mid h$}{b | h}}

If~$b\ne h$ but~$b\mid h$, we have~$\phi_h = 0$ and~$\rho = \tau$. Thus,
Proposition~\ref{prop:Shat} shows that any semigroup with~$x = 0$
has~$\rho\ge\hat\tau$.
Moreover, Proposition~\ref{prop:S*} and Theorem~\ref{thm:E} show that if~$S$
is collision-free, then~$S^{\rightarrow n}$ has at most the same Eliahou
number if~$n$ is large enough.
Accordingly, Eliahou-negative collision-free semigroups may or may not exist,
but if they do, there are infinitely many, possibly only for some congruences
of~$m$ modulo~$b$.

Examples include Bras-Amorós's semigroups~\cite{bras-amoros}, which are~$\hat
S(4,3/2,\{0,1\})^{\rightarrow n}$, with~$m = 19 + 2n$.
They illustrate the dependency on~$m$ modulo~$b$: with~$h
= 4$ and~$a/b = 3/2$, we have~$\hat\tau = 5$.
Since~$E_0(4,3/2,3) = -6$, we need~$\rho = 5$. This forces the
equality~$\theta((h-b)\gamma_1) = 0$, so~$2\gamma_1 + 2m = c = 5m-5$,
so~$m$ is odd.

Using Proposition~\ref{prop:Shat} and since~$\rho = \tau$, we can determine if
Eliahou semigroups with~$x = 0$ exist by checking if the semigroup~$\hat
S(h,a/b,\Delta)$ has~$E < 0$, where~$\Delta$ is a~$B_h$~set of
cardinality~$\ell-1$ with minimal diameter.  We give below a necessary
condition and show that it is sufficient when~$\ell = 3$.

\begin{theorem} \label{thm:Epos}
If~$S$ is collision-free \h-regular, $b = 1$ and either~$\rho = 0$
or~$x = 0$, then~$E\ge0$.
\end{theorem}

We conjecture that this is true for any \h-regular semigroup with~$b = 1$.

\begin{proof}
If~$b = 1$, then~$\phi_i = 0$, so~\eqref{E} shows that~$E = -(h-1)s_h +
\ell x + \rho$.

If~$\rho = 0$, then~$c+m = ham$; every~$\lambda\in i\Gamma$
satisfies~$\lambda < iam$, so we have~$\theta(\lambda) < 0$
and~$x(\lambda) = 1$. Therefore, we have~$x = s_0+\dotsb+s_{h-1} =
(h-1)s_h/\ell$ and~$E = 0$.

If~$x = 0$, then as~$h\Gamma$ contains~$s_h$ elements, we
have~$s_h\le hd+1$. By Proposition~\ref{prop:Shat}, we have~$\rho\ge\hat\tau =
(h-1)(hd+1)\ge(h-1)s_h$, so~$E\ge0$.
\end{proof}

\begin{theorem}
If~$b\mid h$ and~$\fpart{a/b}\not\in\{0,1/h\}$, then~$\hat S(h,a/b,\{0,1\})$
has~$E < 0$.
\end{theorem}

\begin{proof}
With~$\rho = \hat\tau$ and~$\phi_{h-i-jb} = \phi_{h-i}$, we get
from~\eqref{E}:
\[E = -(h-1)s_h +
\ell\sum_{i=0}^{b-1}\left(\phi_{h-i}\sum_{j=0}^{\smash{h/b-1}}s_{i+jb}\right) +
\frac{(h-b)(hd+1)}{b}\text.\]
The numbers~$(\phi_{h-i})$ for $i = 0,\dotsc,b-1$ are a permutation
of~$(i/b)$, so by the rearrangement inequality, the maximum value is reached
if~$\phi_{h-i} = i/b$. With the values~$s_i = \smchoose2i = i+1$, $\ell = 3$
and~$d = 1$, we compute the sums and find:
\[E \le -\frac{(b-1)(h-b)h}{4b}\text.\]
This shows that~$E < 0$ if~$b\ne h$, as we have~$b\ne1$ if~$a/b$ is not an
integer. The case~$b = h$ is given by Theorem~\ref{thm:E0<0}.
\end{proof}

We can do similar computations with~$\ell\ge4$, but the result is not always
negative. However, semigroups with~$x > 0$ can have a value of~$\rho$ less
than~$\hat\tau$ and such a trade-off leads to Eliahou-negative semigroups in
some cases (Figure~\ref{fig:b|h}).
The value~$\rho$ can go all the way to zero: the
semigroup~$\langle71,122,123\rangle_{923}$, which is~$\hat S(8,7/4,\{0,1\})$
with the conductor increased to set~$\rho$ to zero, 
is canonical and has~$E = -6$.

\begin{figure}[ht]%
\begin{tikzpicture}[semigroups]
\sgdraw{59,78,82,83}{274}
\begin{scope}[yshift=-5cm]
\sgdraw{62,82,90,91}{304}
\foreach \x in {46,54,55} \fill[color2] (\x.5,-.5) circle(.2);
\end{scope}
\sglegend4
\end{tikzpicture}
\caption[]{Top: any collision-free $4$\nobreakdashes-regular semigroup
with~$a/b = 3/2$, $\ell = 4$ and~$x = 0$ has~$\rho\ge21$. 
Since~$E_0(4,3/2,4) = -19$, this entails~$E\ge2$.

Bottom: the semigroup~$\langle62,82,90,91\rangle_{304}$ has~$x = 3$
(elements with excess~$1$ are marked with a dot) but~$E = -1$
thanks to the lower~$\rho$.}
\label{fig:b|h}
\end{figure}

\subsection{The case \texorpdfstring{$b\nmid h$}{b ∤ h}}

If~$b\nmid h$, by contrast, we have~$\phi_h\ne0$. Thus, the numbers~$\rho =
\phi_hm+\tau$ and~$E = k\ell - qs + \rho$ scale with~$m$. This shows that
only finitely many Eliahou semigroups exist for any given~$h$, $a/b$
and~$\ell$. No example has been published to our knowledge; the smallest two
are shown in Figure~\ref{fig:b∤h}.

\begin{figure}[htb]%
\begin{tikzpicture}[semigroups]
\sgdraw{19,30,31}{106}
\begin{scope}[yshift=-5cm]
\sgdraw{30, 51, 52}{231}
\end{scope}
\sglegend5
\end{tikzpicture}
\caption[]{Top: the semigroup~$\hat S(4,5/3,\{0,1\}) =
\langle19,30,31\rangle_{106}$, only Eliahou semigroup with~$h = 4$, $a/b =
5/3$ and~$\ell = 3$, with~$E = -1$.

Bottom: the semigroup~$\hat S(5,7/4,\{0,1\}) = \langle30,51,52\rangle_{231}$,
one of seven Eliahou semigroups with~$h = 5$, $a/b = 7/4$ and~$\ell = 3$,
with~$E = -3$.}
\label{fig:b∤h}
\end{figure}

Determining in general for which values of~$h$, $a/b$ and~$\ell$ Eliahou
semigroups exist seems difficult. Checking the semigroup~$\hat S(h, a/b,
\Delta)$ for~$\Delta$ with minimal diameter only determines if split Eliahou
semigroups with~$x = 0$ exist (since~$\rho = \phi_hm + \tau$).
We give below a necessary condition and an infinite family; many
others could be built in the same spirit.

\begin{theorem}
If~$S$ is split and collision-free, $x = 0$ and~$b' = h$, then~$E\ge0$.
\end{theorem}

We conjecture that the split, collision-free and~$x = 0$
conditions are unnecessary.

\begin{proof}
If~$b' = h$, we have~$\phi'_h = 0$ and~$\phi_h = 1 - 1/b$. As in the proof of
Theorem~\ref{thm:E0<0}, the rearrangement inequality applied to~\eqref{E'}
gives~$E \ge-(h-1)s_h + \rho$. With~$s_h\le hd+1$ and~$\rho \ge(1-1/b)\hat
m + \hat\tau$, this boils down to~$E\ge b-1$.
\end{proof}

\begin{theorem}
If~$b > \ceil{h/2}$, the semigroup~$\hat S\sparen{h,\frac{2b-1}{b},\{0,1\}}$
has Eliahou number:
\[E = -\binom{b-1}{2}\text.\]
\end{theorem}

In this case, we have~$\phi_h = h/b - 1$. As~$\rho/m>\phi_h$, by
taking~$b = \lceil h/2\rceil + 1$, we
get Eliahou semigroups with the ratio~$\rho/m$ arbitrarily close to~$1$.

\begin{proof}
We have~$\phi_1 = 1/b$ and the condition~$b > \ceil{h/2}$ implies
that~$\floor{h\phi_1} = 1$, so we use \eqref{Ea}. We have~$s_h = h+1$ and~$s =
\sbinom{h+2}{2}$.
Moreover, since~$2b-1 > h$, the
\h-Farey predecessor is~$\smash{\frac{a'}{b'}} = \smash{\frac{2b-3}{b-1}}$.
With~$\rho = (h/b-1)\hat m+\hat\tau$,
this boils down to~$E = -\sbinom{b-1}{2}$.
\end{proof}

\subsection{Beyond collision-free \texorpdfstring{$h$-}{h-}regular semigroups}
\label{sec:beyond}

As shown in~\cite{delgado,eliahou-fromentin}, the Eliahou number can
get arbitrarily low. This leaves room for semigroups to deviate from the
patterns that we presented while still having~$E < 0$: unnecessarily
large~$\rho$ and~$x$, collisions, semigroups not quite~\h-regular, etc.
This permits semigroups with a conductor smaller than otherwise
possible (Figure~\ref{fig:beyond}).

Eventually, larger deviations can occur. For example, the semigroup:
\[\hat S(3,5/3,\{0,1,15,18,23\})^{\rightarrow24} = \langle284,450,451,465,468,473\rangle_{1136}\]
is collision-free $3$\nobreakdashes-regular and has~$E = -20$. Adding
an extra out-of place left generator, we get:
\[\langle284,450,451,465,468,473,876\rangle_{1136}\text,\]
a collision-free semigroup with~$6$ fewer right generators and~$E = -3$.
We do not know if there are Eliahou semigroups truly far from being
\h-regular, for instance with left generators belonging to several distinct
intervals.

\begin{figure}[ht]%
\begin{tikzpicture}[semigroups]
\sgdraw{55,82,85,90,91}{219}
\begin{scope}[yshift=-5cm]
\sgdraw{58, 84, 91, 95, 96}{232}
\fill[color2] (52.5,-.5) circle(.2);
\end{scope}
\begin{scope}[yshift=-10cm]
\sgdraw{59, 88, 90, 95, 99}{235}
\end{scope}
\sglegend3
\end{tikzpicture}
\caption[]{Three $3$\nobreakdashes-regular semigroups with~$a/b = 5/3$ (which
implies~$q = 4$) and~$\ell = 5$. As~$E_0(3,5/3,5) = -10$, an element with
excess~$1$ increases~$E$ by~$5$ and an extra right generator by~$4$, so they
all have~$E = -1$.

Top: the semigroup~$\langle55,82,85,90,91\rangle_{219}$ has two collisions.

Middle: the semigroup~$\langle58,84,91,95,96\rangle_{232}$ has one collision
and excess~$1$.

Bottom: the semigroup~$\langle59,88,90,95,99\rangle_{235}$ has one collision
and is not quite~$3$\nobreakdashes-regular: one element of~$3\Gamma$ is
beyond~$c+m$.}
\label{fig:beyond}
\end{figure}

Almirón and Moyano-Fernández~\cite{almiron-moyano-fernandez} found~$40$
nearly collision-free $4$\nobreakdashes-regular semigroups
with~$a/b\in\{7/4,11/4\}$, $\ell = 4$, $m = 100$ and~$\rho\le3$. Some have one
element of~$4\Gamma$ beyond~$c+m$, some have one collision.  They have an
Eliahou number ranging from~$-1$ to~$-8$ (the value~$E_0(4,7/4,4)$ is~$-14$).

\section{Experimental search} \label{sec:experimental}

We finish the paper by presenting the algorithms we used to get our set of
Eliahou-negative semigroups and check the Wilf conjecture for~$c\le200$.

\subsection{Double elements algorithm}

The algorithm we used is a variant of the one presented in~\cite{bras-amoros}
to explore the numerical semigroup tree, in which one goes from parent to
child by removing a right generator from the semigroup. Let~$L = S\cap[0,c)$
be the set of left elements of~$S$ and~$D = L + L$ the set of \emph{double
elements}. Since the set of right generators is~$[c,c+m)\setminus D$,
knowing~$D$ is enough to know which elements we may remove.

In the exploration of the tree, we consider semigroups of the form~$S = \langle
m,\Gamma\rangle_c$. We do not assume that~$S$ is canonical;
the rank and other parameters are counted according to this given value
of~$c$, not the real conductor of~$S$. Let~$S'$ be a semigroup with
conductor~$c' = c+1$ (and rank~$k'$, etc.) and agreeing with~$S$
on~$\{0,\dotsc,c-1\}$. There are three possibilities.
\begin{itemize}
\item The point~$c$ is in~$S'$ but not a generator. This is only possible
if~$c\in D$. In this case, we have~$L' = L\cup\{c\}$ and~$D' =
D\cup(c+L')$. Since~$c\in D$ and~$c+m\in D'$, we have~$r' = r$.
\item The point~$c$ is a gap of~$S'$ ($S'$ is canonical). This is only
possible if~$c\notin D$.  In this case, we have~$L' = L$ and~$r' = r-\mathbb
1_{c+m\in D}$.
\item The point~$c$ is a left generator of~$S'$. This is also only possible
if~$c\notin D$. In this case, we have~$L' = L\cup\{c\}$ and~$D' =
D\cup(c+L')$. Since~$c\notin D$ and~$c+m\in D'$, we have~$r' = r-1$.
\end{itemize}
In essence, we have replaced the usual numerical semigroup tree, which has
arity bounded by~$m$, by a unary-binary tree by adding some intermediary
nodes (the non-canonical semigroups). This tree is illustrated in
Figure~\ref{fig:tree}.

\colorlet{colorl}{bred}
\colorlet{colord}{bblue}
\colorlet{colorr}{bgrey}
\begin{figure}[ht]
\begin{tikzpicture}[box semigroups]
\foreach \x in {0,4} { \sgbox l{\x} }
\foreach \x in {8} \sgbox d{\x}
\foreach \x in {6,7,9,10} { \sgbox r{\x} }
\begin{scope}[xshift=15cm,yshift=1.5cm]
\foreach \x in {0,4} { \sgbox l{\x} }
\foreach \x in {8} { \sgbox d{\x} }
\foreach \x in {7,9,10,11,12} { \sgbox r{\x} }
\end{scope}
\begin{scope}[xshift=15cm,yshift=-1.5cm]
\foreach \x in {0,4,6} { \sgbox l{\x} }
\foreach \x in {8,10,12} { \sgbox d{\x} }
\foreach \x in {7,9,11} { \sgbox r{\x} }
\end{scope}
\begin{scope}[yshift=4.5cm]
\foreach \x in {0,3} { \sgbox l{\x} }
\foreach \x in {6} { \sgbox d{\x} }
\foreach \x in {7,8,9,10} { \sgbox r{\x} }
\end{scope}
\begin{scope}[xshift=15cm,yshift=4.5cm]
\foreach \x in {0,3,6} { \sgbox l{\x} }
\foreach \x in {9,12} { \sgbox d{\x} }
\foreach \x in {7,8,10,11} { \sgbox r{\x} }
\end{scope}
\draw[-latex] (11.5,.25) -- ++(2,1);
\draw[-latex] (11.5,-.25) -- ++(2,-1);
\draw[-latex] (11.5,4.5) -- ++(2,0);
\end{tikzpicture}
\caption[]{%
Top: the semigroup~$\langle3\rangle_6$ has one child in the tree,
non-canonical~$\langle3\rangle_7$.

Bottom: the semigroup $\langle4\rangle_6$ has two children,
$\langle4\rangle_7$ and non-canonical~$\langle4,6\rangle_7$.

Left elements are in red and double elements in blue.}
\label{fig:tree}
\end{figure}

To explore this tree, we track the following variables, with
Algorithm~\ref{algo:update} showing the procedures to update them:
\begin{itemize}
\item the integers~$c_S$, $k_S$, $\ell_S$ and~$r_S$
representing~$c$, $k$, $\ell$ and~$r$;
\item the bitfields~$\mathcal L_S$ and~$\mathcal D_S$ representing~$(L-1)\cap\mathbb
N$ and~$(D-c)\cap\mathbb N$, subsets of~$\{0,\dotsc,c-2\}$.
\end{itemize}

\begin{algorithm}[ht]
\caption{Going from~$c$ to~$c+1$} \label{algo:update}
\smallskip
\begin{minipage}[t]{.495\textwidth}
\begin{algorithmic}[1]
\Require $0\in \mathcal D_S$
\Procedure{add\_non\_gen}{$S$}
  \State $\mathcal L_S\gets \mathcal L_S\cup\{c_S-1\}$
  \State $\mathcal D_S\gets (\mathcal D_S\setminus\{0\} - 1)\cup \mathcal L_S$
  \State $c_S\gets c_S + 1$
  \State $k_S\gets k_S + 1$
\EndProcedure
\end{algorithmic}
\begin{algorithmic}[1]
\Require $0\notin \mathcal D_S$
\Procedure{add\_gap}{$S$}
  \State $\mathcal D_S\gets \mathcal D_S - 1$
  \State $c_S\gets c_S + 1$
  \If{$m-1\in \mathcal D_S$}
    \State $r_S\gets r_S - 1$
  \EndIf
\EndProcedure
\Statex
\end{algorithmic}
\end{minipage}
\begin{minipage}[t]{.495\textwidth}
\begin{algorithmic}[1]
\Require $0\notin \mathcal D_S$
\Procedure{add\_left\_gen}{$S$}
  \State $\mathcal L_S\gets \mathcal L_S\cup\{c_S-1\}$
  \State $\mathcal D_S\gets (\mathcal D_S - 1) \cup \mathcal L_S$
  \State $c_S\gets c_S + 1$
  \State $k_S\gets k_S + 1$
  \State $\ell_S\gets \ell_S + 1$
  \State $r_S\gets r_S - 1$
\EndProcedure
\end{algorithmic}
\end{minipage}\vspace{-.75em}%
\end{algorithm}

We can easily compute~$q = \ceil{c/m}$ and thus~$E$ from these variables.
Moreover, by starting with semigroups of the form~$\langle m\rangle_{m+1}$,
the multiplicity~$m$ is kept constant, which facilitates checking if~$m-1\in
\mathcal D_S$ and enables~$q$ to be tracked without a division.

The reason we track the sets~$\mathcal L_S$ and~$\mathcal D_S$ is to check
if~$0$ and~$m-1$ are in~$\mathcal D_S$. If we are interested in semigroups
with~$c \le c_{\max}$, it follows that we only need the lowest~$c_{\max} - c +
m$ bits of both sets. As they have~$c-1$ bits, we only need~$(c_{\max} + m -
1)/2$ bits.
It may seem that we lose some information on the semigroup this way, but as
explained below, we use a trick to retain full knowledge of it.

\subsection{Depth-first traversal and branch pruning}

Following~\cite{fromentin-hivert,bras-amoros-fernandez,bras-amoros}, we
perform a depth-first search of the tree, which uses the least amount of
memory. Since the tree grows impractically fast with~$c$, we need to prune
branches that we know do not contain Eliahou-negative semigroups. Consider~$S
= \langle m,\Gamma\rangle_c$ and~$S' = \langle
m,\Gamma,\gamma_{\ell},\dots,\gamma_{\ell'-1}\rangle_c$ with~$\ell'>\ell$.
Let~$E$ and~$E'$ be the Eliahou numbers of~$S$ and~$S'$, respectively.

\begin{lemma} \label{lem:prune}
If~$(k+1)(\ell+1) \ge c$, then~$E'\ge0$.
\end{lemma}

\begin{proof}
This is clear from the facts that~$k'\ge k+1$ and~$\ell'\ge\ell+1$.
\end{proof}

\begin{lemma} \label{lem:gamma1}
If~$\gamma_{\ell} \ge c + m - \gamma_1$, then~$E'\ge E$.
\end{lemma}

\begin{proof}
Every extra generator~$\gamma_i$ for~$i = \ell,\dotsc,\ell'-1$ can remove at 
most one right generator of the form~$\gamma_i + jm$ since~$\gamma_i +
\gamma_1 > c+m$.
From~\eqref{Er}, we get~$E'\ge E + (\ell'-\ell)(k-q) \ge E$.
\end{proof}

The exploration can be sped up if we assume some of the four conjectures
below.
Proving these conjectures, maybe in some weaker form, would
enable us to rigorously find all Eliahou-negative semigroups for a higher
conductor.

\begin{conjecture} \label{conj1}
If~$\gamma_{\ell}\ge c + m - 2\gamma_1$ and~$E\ge0$, then~$E'\ge0$.
\end{conjecture}

In this case, as~$\gamma_\ell + \gamma_i + \gamma_j\ge c+m$ for~$i,j\ge1$,
adding~$\gamma_{\ell}$ can subtract up to~$\ell$ right generators.
We do not have~$E'\ge E$ in general:
for example, $S = \langle23,39,40,42\rangle_{77}$ has~$E = 15$ but~$S' =
\langle23,39,40,42,54\rangle_{77}$ has~$E' = 14$. We hope that this does
not happen if~$E < 0$ or at least not unless~$S$ is very large.

\begin{conjecture} \label{conj2}
If~$r < \ell$, then~$E\ge0$.
\end{conjecture}

\begin{conjecture} \label{conj3}
If~$\gamma_1\le m + 6$, then~$E\ge0$.
\end{conjecture}

\begin{conjecture} \label{conj4}
If~$c \le 4m - \binom{\ell}{3}$, then~$E\ge0$.
\end{conjecture}

The only semigroup we found with~$\gamma_1 = m + 7$
is~$\langle19,26,27\rangle_{90}$, the first Bras-Amorós
semigroup~\cite{bras-amoros}. The value~$c = 4m - \smash{\binom{\ell}{3}} + 1$
is reached by the Eliahou--Fromentin semigroups~\cite{eliahou-fromentin}. We
believe both bounds to be optimal.
Lemma~\ref{lem:sidon} suggests that Eliahou-negative semigroups have an~$r$ at
least quadratic in~$\ell$.

Using Lemma~\ref{lem:prune}, we explore the descendants of a
semigroup~$\langle m,\Gamma\rangle_{\gamma_{\ell-1}+1}$ with conductor in a
range~$[c_{\min},c_{\max}]$ in two stages.
\begin{itemize}
\item First, we build all semigroups~$\langle m,\Gamma\rangle_{c}$
for~$c\in[c_{\min},c_{\max}]$. We output any Eliahou-negative semigroups
found and compute~$[c_{\min}',c_{\max}']$, the smallest range containing the
canonical semigroups with~$(k+1)(\ell+1) < c$.
\item Second, for all suitable~$\gamma_\ell$, we explore the
descendants of~$\langle m,\Gamma,\gamma_\ell\rangle_{\gamma_\ell+1}$ 
recursively in the range~$[c_{\min}',c_{\max}']$.
\end{itemize}
The branch pruning is thus achieved by the decreasing value of~$c_{\max}$. If
no semigroup is found with~$(k+1)(\ell+1) < c$ in the first stage, the branch
is eliminated altogether.

Using Lemma~\ref{lem:gamma1}, we restrict the search to semigroups
with~$\gamma_{\ell-1} < c + m - \gamma_1$. This is described in detail
in Algorithm~\ref{algo:explore}. To facilitate this, we start from semigroups
of the form~$\langle m,\gamma_1\rangle_{\gamma_1+1}$, so the
value~$\gamma_1$ remains constant; using Theorem~\ref{thm:h3} and the fact,
due to Eliahou~\cite{eliahou}, that any semigroup with~$c\le3m$ is
Eliahou-positive, we use~$c_{\min} = \max(3m,3\gamma_1-m)+1$. After the
exploration, we try adding generators above~$c+m-\gamma_1$ to Eliahou-negative
semigroups to find new ones; given their rarity, this takes negligible time.

\begin{algorithm}[ht]
\caption{Exploring the descendants of a semigroup} \label{algo:explore}
\begin{algorithmic}[1]
\Ensure Eliahou-negative descendants of~$S$ with~$c\in[c_{\min},c_{\max}]$
and~$\gamma_{\ell-1} < c+m-\gamma_1$
\Procedure{explore}{$S,c_{\min},c_{\max}$}
 \State $S'\gets S$
 \State{$c_{\min}'\gets\infty$}
 \State{$c_{\max}'\gets 0$}
 \While{$c_{S} < c_{\max}$}
  \If{$0\in \mathcal D_{S}$}
   \State \Call{add\_non\_gen}{$S$}
  \Else
   \State \Call{add\_gap}{$S$}
   \If{$c_S\ge c_{\min}$ \textbf{and} $k_{S}\ell_{S} + \lceil c_{S}/m\rceil r_{S} < c_{S}$}
    \State \Output{$S$}
   \EndIf
   \If{$c_S\ge c_{\min}$ \textbf{and} $(k_{S}+1)(\ell_{S}+1) < c_{S}$} \label{line:kl}
    \State{$c_{\min}'\gets \min(c_{\min}',c_{S})$}
    \State{$c_{\max}'\gets c_{S}$}
   \EndIf
  \EndIf
 \EndWhile

 \While{$c_{S'} + \gamma_1 - m < c_{\max}'$} \label{line:gamma}
  \If{$0\in \mathcal D_{S'}$}
   \State \Call{add\_non\_gen}{$S'$}
  \Else
   \State $c_{\min}'\gets\max(c_{\min}', c_{S'} + \gamma_1 - m + 1)$
   \label{line:gamma2}
   \State $S\gets S'$
   \State \Call{add\_gap}{$S'$}
   \State \Call{add\_left\_gen}{$S$}
   \State \Call{explore}{$S,c_{\min}',c_{\max}'$}
  \EndIf
 \EndWhile
\EndProcedure
\end{algorithmic}
\end{algorithm}

Note that each recursive call is done with a higher value
of~$\ell$, so the condition~$(k+1)(\ell+1) < c$ keeps the size of the calling
stack under~$\sqrt{c_{\max}}$. Moreover, the left generators of~$S$ consist of
the conductors of all semigroups in the stack, so we can recover full
knowledge of~$S$ by looking into the stack.  Alternatively, one could track
the list of left generators of~$S$ in a separate variable.

If we assume Conjecture~\ref{conj1}, we change~$\gamma_1-m$
to~$2\gamma_1-m$ in Lines~\ref{line:gamma} and~\ref{line:gamma2}.
If we assume Conjecture~\ref{conj2}, we change the
product~$(k_S + 1)(\ell_S + 1)$ in Line~\ref{line:kl} to~$(k_S + \ceil{c_S/m}
+ 1)(\ell_S + 1)$. Conjectures~\ref{conj3} and~\ref{conj4} let us shorten the
list of semigroups~$\langle m,\gamma_1\rangle_{\gamma_1+1}$ we need to
explore and use a greater initial value of~$c_{\min}$.

\subsection{Practical implementation and results} \label{sec:practical}

Following again~\cite{fromentin-hivert}, we implemented our algorithm by
taking advantage of the computer architecture%
\footnote{See
\url{https://depot.lipn.univ-paris13.fr/bacher/eliahou-semigroups} for the
code and its outputs.}%
. We store the sets~$\mathcal L_S$
and~$\mathcal D_S$ in 256-bit registers and operate on them using instructions
in the AVX2 set. The right rotation of a register, used to compute~$\mathcal
D_S - 1$, can be computed in four instructions (\verb!vpermq!, \verb!vpsrlq!,
\verb!vpsllq!, \verb!vpor!). The register representing the set~$\{c_S-1\}$ can
be computed in one instruction (\verb!vpsllvd!) from a vector
representing~$c_S$ which can be incremented in one instruction
(\verb!vpaddd!). All these assembly instructions can be accessed in C using
intrinsic functions.

We handled the parallel computations with a simple MPI distributed program. We
ran three experiments, all taking a few hours on~$10$ machines with~$40$
cores:
\begin{itemize}
\item for $c\le200$ assuming no conjectures, which yielded~$778$
semigroups;
\item for $c\le320$ assuming Conjecture~\ref{conj1}, which
yielded~$17064$ semigroups;
\item for $c\le400$ assuming Conjectures~\ref{conj1}, \ref{conj2}, \ref{conj3}
and~\ref{conj4} and restricting the search to~$\ell\le7$, which
yielded~$104569$ semigroups, all with~$\ell\le6$.
\end{itemize}
In all three cases, we have~$(c_{\max} + m - 1)/2 \le 256$, so
using~256-bit registers is enough (in the third case, this is true because we
assumed Conjecture~\ref{conj4}).
Adding generators greater than~$c+m-2\gamma_1$ did not yield any extra
semigroups, but it does eventually happen. For example, the semigroup
mentioned in Section~\ref{sec:beyond}:
\[\langle284,450,451,465,468,473,876\rangle_{1136}\]
has~$E = -3$ but satisfies Conjecture~\ref{conj1}.

We checked that all the semigroups we found are Wilf-positive, so Wilf's
conjecture holds for~$c\le200$.

Table~\ref{tab:results} shows how the semigroups with~$c\le400$ fit into the
classes defined in the paper. To keep it legible, we excluded the~$1886$
semigroups that can be written as~$S^\uparrow$ for~$S$ a smaller semigroup.

\begin{table}[ht]
\setlength\tabcolsep{3.5pt}%
\begin{tabular}{l>{\raggedleft\arraybackslash}p{\widthof{34099}}*3{>{\raggedleft\arraybackslash}p{\widthof{366}}}*4{>{\raggedleft\arraybackslash}p{\widthof{147}}}*6{>{\raggedleft\arraybackslash}p{\widthof{25}}}>{\raggedleft\arraybackslash}p{\widthof{$\dfrac97$}}}
$h$&\multicolumn1c3&\multicolumn3c4&\multicolumn4c5&\multicolumn6c6&\multicolumn1c7\\
\cmidrule(r){1-1}\cmidrule(lr){2-2}\cmidrule(lr){3-5}\cmidrule(lr){6-9}\cmidrule(lr){10-15}\cmidrule(lr){16-16}
$\dfrac ab$\vphantom{$\Biggl($}&
\titlefrac53&
\titlefrac32&
\titlefrac53&
\titlefrac74&
\titlefrac75&
\titlefrac85&
\titlefrac74&
\titlefrac95&
\titlefrac43&
\titlefrac32&
\titlefrac85&
\titlefrac53&
\titlefrac74&
\titlefrac95&
\titlefrac97\\\midrule
CF, split&%
\sgbg{3,4}{12893}&\sgbg3{32}&\sgbg3{1}&\sgbg3{366}&\sgbg3{132}&\sgbg3{110}&\sgbg3{7}&\sgbg3{25}&\sgbg3{25}&\sgbg3{12}&\sgbg3{1}&\sgbg3{2}&\sgbg3{3}\\
CF, $x = 0$&%
\sgbg{3,4,5}{22716}&&&\sgbg{3,4}{188}&\sgbg3{117}&\sgbg3{147}&&\sgbg3{13}&&&&&&&\sgbg3{1}\\
CF, $x>0$&\sgbg5{5471}&\sgbg{4}{18}&&\sgbg{4}{11}&&&&\sgbg{3}{6}&&&&\sgbg{3}{1}&&\sgbg{3}{1}\\
collisions&\sgbg{5,6}{42938}&&&\sgbg{4}{112}\\
not \h-regular&\sgbg{5,6}{17276}&&&\sgbg{4}{58}
\end{tabular}\\\bigskip
\begin{tabular}{*3{p{5mm}p{15mm}}p{5mm}l}
\sgbg3{}&$\ell = 3$&\sgbg4{}&$\ell = 4$&\sgbg5{}&$\ell = 5$&\sgbg6{}&$\ell = 6$
\end{tabular}
\caption{The repartition of the semigroups with~$c\le400$,
sorted according to~$h$, $a/b$ and membership in the classes defined in the
paper (CF means collision-free and ``$x = 0$'' means $x = 0$ but not
split). Non \h-regular semigroups are all nearly \h-regular with the
parameters given. The background color(s) show the number of left generators
of the semigroups.}
\label{tab:results}
\end{table}

\section*{Acknowledgments}

I am grateful to Maria Bras-Amorós for helpful discussions during the research
that led to this paper and to Johan Mercedes for pointing out the
reference~\cite{kovacevic} and how it can be used to show Wilf-positivity.

The experiments in this work were carried out on Magi, the experimental
platform of Université Sorbonne Paris Nord (USPN) dedicated to research.
This platform offers researchers at the institution High-Performance Computing
(HPC), cloud and storage services.

\bibliographystyle{abbrv}
\bibliography{biblio}{}

\end{document}